\documentclass[12pt]{amsart}
\usepackage{geometry}
\usepackage{amssymb}
\usepackage{amsmath}
\usepackage{amsthm}
\usepackage{stmaryrd}
\usepackage{xcolor}

\theoremstyle{plain}
\newtheorem{theorem}{Theorem}[section]

\newtheorem{corollary}[theorem]{Corollary}
\newtheorem{proposition}[theorem]{Proposition}
\newtheorem{remark}[theorem]{Remark}
\newtheorem{example}[theorem]{Example}

\begin{document}

\title{The ranks of twists of an elliptic curve in characteristic $3$}
\author{João Paulo Guardieiro}
\address{Instituto de Ciências  Matemáticas e Computação, Av. Trabalhador São Carlense 400, 13566-590 São Carlos, Brazil.}
\address{Bernoulli Institute,  Nijenborgh 9,
9747 AG~Groningen, the Netherlands.}
    \email{joaopaulosousa20@gmail.com}

\maketitle

\begin{center}
    \textbf{Abstract}
    
    Starting from the elliptic curve $E: y^2 = x^3 - x$ over $\mathbb{F}_9$, a curve $\mathcal{X}$ over $\mathbb{F}_{3^{2n}}$ and a cyclic cover of $\mathcal{X}$ of degree $m \in \{2,3,4,6\}$, we construct the corresponding $m$-twists over the function field $\mathbb{F}_{3^{2n}}(\mathcal{X})$. We also obtain the Mordell-Weil rank of these twists in terms of the Zeta functions of $\mathcal{X}$ and of suitable Kummer and Artin-Schreier extensions of it. Finally, we also describe the fibers of the elliptic fibration associated to such twists in the case $\mathcal{X} = \mathbb{P}^1$.
\end{center}

\section{Introduction}
The theory of elliptic curves is a very interesting topic in mathematics. Since it is possible to define a group structure on their sets of points (see \cite[Algorithm III.2.3]{silverman}), these curves have many applications in Number Theory and Cryptography (see \cite{washington}). In $1922$, Mordell proved that the set of $\mathbb{Q}$-rational points of an elliptic curve $E$ over $\mathbb{Q}$ is a finitely generated abelian group (see \cite{mordell}), that is,
$$E(\mathbb{Q}) = E(\mathbb{Q})_{\textrm{Tor}} \oplus \mathbb{Z}^r,$$
with $E(\mathbb{Q})_{\textrm{Tor}}$ finite. The integer $r$ is called the (Mordell-Weil) rank of $E$, and it is still an open problem if this number can be arbitrarily large (the current record is $29$, and was given by Elkies and Klagsbrun in $2024$). 

Mordell's theorem generalizes to function fields over a finite field, and the question whether there exist elliptic curves with arbitrarily large rank over such fields has a positive answer: in $1967$, Tate and Shafarevich constructed families of elliptic curves over $\mathbb{F}_{p^n}(t)$ with odd $p$ with this property by considering the quadratic twist of an elliptic curve over $\mathbb{F}_p$ in the function field of a hyperelliptic curve (see \cite{TateSha67}). This construction was then used by Elkies in $1994$ to give a positive answer for elliptic curves over $\mathbb{F}_{2^n}(t)$ (see \cite{elkies}). 

Recently, Bootsma, Tafazolian and Top also obtained elliptic curves of arbitrarily large rank over $\mathbb{F}_q(t)$, but now considering quartic and sextic twists in the function field of a maximal curve (for $q \equiv 3$ (mod $4$) and $q \equiv 5$ (mod $6$) respectively) and a decomposition of the set of rational points into eigenspaces (see \cite{BTT}). Finally, the author, in a joint work with Borges, Salgado and Top, studied quartic twists in the function field of a maximal curve over fields of even characteristic to obtain a result analogous to that of Elkies (see \cite{BGST}). The present paper completes and unifies these results.

When studying elliptic curves, it is common to assume that the characteristic of the field is different from $2$ and $3$, and results in these characteristics are frequently less known. In this work we will then focus on the twists of a given supersingular elliptic curve in characteristic $3$. Analogously to what is done in \cite{BTT} and \cite{BGST}, we will consider a decomposition into eigenspaces to obtain rank formulas for the twists we will construct, but we will drop the maximality condition on the curves used on their construction. Also, the quartic twist we obtain in this paper extends the construction in \cite{TateSha67}, since we consider this twist in the function field of any degree $2$ extension of a curve $\mathcal{X}$, while the original work considers $\mathcal{X} = \mathbb{P}^1$. For completeness, even though some results of this work are similar to known ones, some details of their proofs will be presented, to indicate where minor changes from the previous cases should be applied.

This paper is structured in the following way: in Section \ref{cubic-twists} we first study the cubic twists of $E$ using an Artin-Schreier extension, derive their rank formula, present a family of elliptic curves having arbitrarily large rank and describe the only singular fiber of the associated elliptic fibration. In Section \ref{other-twists} we discuss the other twists of $E$, namely the quadratic, quartics and sextics ones, also presenting their rank formulas and discussing the fibers of the associated elliptic fibration over $\mathbb{P}^1$ defined by them.

The main results of this work are the following.

\vspace{0.3cm}

\textbf{Theorem A.} Let $E / \mathbb{F}_9: y^2 = x^3 - x$ and let $\mathcal{X}$ be a curve defined over $\mathbb{F}_{3^{2n}}$. Then $E$, considered over $\mathbb{F}_{3^{2n}}(\mathcal{X})$, admits the following twists.
\begin{itemize}
	\item Quadratic twist: $E_2: \eta^2 = \xi^3 - f^2\xi$;
	\item Cubic twists: $E_{3, \tau}: \eta^2 = \xi^3 - \xi - g$ and $E_{3, \tau^2}: \eta^2 = \xi^3 - \xi + g$;
	\item Quartic twists: $E_{4, \tau}: \eta^2 = \xi^3 - f\xi$ and $E_{4, \tau^3}: \eta^2 = \xi^3 - f^3\xi$;
	\item Sextic twists: $E_{6, \tau}: \eta^2 = \xi^3 - f^2\xi + f^3g$ and $E_{6, \tau^5}: \eta^2 = \xi^3 - f^2\xi - f^3g$,
\end{itemize}
where $f, g \in \mathbb{F}_{3^{2n}}(\mathcal{X}) \backslash \mathbb{F}_{3^{2n}}$, $f$ is not a square and $g \neq h^3 - h$ for any $h \in \mathbb{F}_{3^{2n}}(\mathcal{X})$.

\vspace{0.3cm}

\textbf{Theorem B.} Let $\mathcal{X}$ be a curve defined over $\mathbb{F}_{3^{2n}}$ and consider the following coverings of $\mathcal{X}$.
$$\mathcal{C}: w^3 - w = g, \quad \mathcal{D}: s^2 = f, \quad \mathcal{G}: r^4 = f$$
and
$$\mathcal{H}: \left\{
\begin{array}{ll}
	s^2 & = f \\
	w^3 - w & = g
\end{array}
\right.,$$
where $f, g \in \mathbb{F}_{3^{2n}}(\mathcal{X}) \backslash \mathbb{F}_{3^{2n}}$, $f$ is not a square and $g \neq h^3 - h$ for any $h \in \mathbb{F}_{3^{2n}}(\mathcal{X})$. Denote by $L_E(T)$ the $L$-function of the elliptic curve $E$ over $\mathbb{F}_{3^{2n}}$. Writing the $L$-functions over $\mathbb{F}_{3^{2n}}$ of $\mathcal{X}$, $\mathcal{C}$, $\mathcal{D}$, $\mathcal{G}$ and of $\mathcal{H}$ as, respectively,
$$L_{\mathcal{X}}(T) = L_{E}(T)^{m(\mathcal{X})} \cdot X(T), \quad L_{\mathcal{C}}(T) = L_{E}(T)^{m(\mathcal{C})} \cdot C(T),$$
$$L_{\mathcal{D}}(T) = L_{E}(T)^{m(\mathcal{D})} \cdot D(T), \quad L_{\mathcal{G}}(T) = L_{E}(T)^{m(\mathcal{G})} \cdot G(T)$$
and
$$L_{\mathcal{H}}(T) = L_{E}(T)^{m(\mathcal{H})} \cdot H(T),$$
and considering the twists of $E: y^2 = x^3 - x$ over $\mathbb{F}_{3^{2n}}(\mathcal{X})$ given in Theorem A, we have
\begin{itemize}
	\item $\textrm{rank }E_2(\mathbb{F}_{3^{2n}}(\mathcal{X})) = 4[m(\mathcal{D}) - m(\mathcal{X})]$;
	\item $\textrm{rank }E_{3, \tau}(\mathbb{F}_{3^{2n}}(\mathcal{X})) = \textrm{rank }E_{3, \tau^2}(\mathbb{F}_{3^{2n}}(\mathcal{X})) = 2[m(\mathcal{C}) - m(\mathcal{X})]$;
	\item $\textrm{rank}\ E_{4, \tau}(\mathbb{F}_{3^{2n}}(\mathcal{X})) = \textrm{rank}\ E_{4, \tau^3}(\mathbb{F}_{3^{2n}}(\mathcal{X})) = 2[m(\mathcal{G}) - m(\mathcal{D})]$;
	\item $\textrm{rank}\ E_{6, \tau}(\mathbb{F}_{3^{2n}}(\mathcal{X})) = \textrm{rank}\ E_{6, \tau^5}(\mathbb{F}_{3^{2n}}(\mathcal{X})) = 2[m(\mathcal{H}) - m(\mathcal{C}) - m(\mathcal{D}) + m(\mathcal{X})]$.
\end{itemize}

\vspace{0.3cm}

\textbf{Theorem C.} Given any natural number $M$, there is $q$, which is a power of $3$, and an elliptic curve over $\mathbb{F}_q(t)$ such that its rank over $\mathbb{F}_q(t)$ equals $2 \cdot \sqrt{q}>M$.

\begin{center}
    \textbf{Notation}
\end{center}

We now fix some notation we will use throughout the text:
\begin{itemize}
\item $\mathbb{F}_{3^{2n}}$ is the finite field with $3^{2n}$ elements and characteristic $3$;
\item $E$ is the elliptic curve defined by $y^2 = x^3 - x$ over $\mathbb{F}_9$;
\item $\Phi_{u,t}$ is the  automorphism of $E$ given by $(x, y) \mapsto (u^2x + t, u^3y)$, with $u^4 = 1$ and $t \in \mathbb{F}_3$;
\item $\mathcal{X}$ is a curve over $\mathbb{F}_{3^{2n}}$ and $\mathbb{F}_{3^{2n}}(\mathcal{X})$ denotes its function field with coefficients in $\mathbb{F}_{3^{2n}}$;
\item $\mathcal{C}$, $\mathcal{D}$, $\mathcal{G}$ and $\mathcal{H}$ are coverings of $\mathcal{X}$ given by
$$\mathcal{C}: w^3 - w = g, \quad \mathcal{D}: s^2 = f, \quad \mathcal{G}: r^4 = f$$
and
$$\mathcal{H}: \left\{
\begin{array}{ll}
	s^2 & = f \\
	w^3 - w & = g
\end{array}
\right.,$$
where $f, g \in \mathbb{F}_{3^{2n}}(\mathcal{X}) \backslash \mathbb{F}_{3^{2n}}$, $f$ is not a square and $g \neq h^3 - h$ for any $h \in \mathbb{F}_{3^{2n}}(\mathcal{X})$.
\end{itemize}

\section{Preliminaries}\label{preliminaries}

In this section, we present some details on the elliptic curve we will use to obtain its twists and the theorem from Tate and Shafarevich we will use to obtain their Mordell-Weil ranks. 

One can check that the elliptic curve $E: y^2 = x^3 - x$ has $4$ $\mathbb{F}_3$-rational points, so the Frobenius endomorphism on $E$ given by $F(x, y) = (x^3, y^3)$ satisfies $F^2 = [-3]$, the multiplication by $-3$ map (see \cite[Theorem 2.3.1]{silverman}). As a consequence, $E$ is a maximal curve over $\mathbb{F}_9$. Therefore, one concludes that $E$ is maximal (respectively, minimal) over odd (respectively, even) degree extensions of $\mathbb{F}_9$. The $L$-function of $E$ over $\mathbb{F}_{3^{2n}}$ is then given by
$$L_{E/\mathbb{F}_{3^{2n}}}(T) = (3^nT + (-1)^{n + 1})^2.$$

Moreover, the automorphism group of $E$ is composed of $12$ automorphisms given by
$$\Phi_{u,t} : (x, y) \mapsto (u^2x + t, u^3y), \quad \textrm{with } u^4 = 1 \textrm{ and } t \in \mathbb{F}_3$$
(see \cite[Proposition A.1.2]{silverman}). Let $i \in \mathbb{F}_9$ be such that $i^2 = -1$. A straightforward computation shows that
\begin{equation}\label{morphisms-char3}
	\textrm{ord}(\Phi_{u, t}) = \left\{\begin{array}{rl}
		1 & , \textrm{ if } (u, t) = (1, 0) \\
		2 & , \textrm{ if } (u, t) = (-1, 0) \\
		3 & , \textrm{ if } (u, t) \in \{(1, 1), (1, -1)\} \\
		4 & , \textrm{ if } (u, t) \in \{(i, t), (-i, t);\ t \in \mathbb{F}_3\} \\
		6 & , \textrm{ if } (u, t) \in \{(-1, 1), (-1, -1)\} \\
	\end{array}\right..
\end{equation}
These automorphisms will be used to obtain the equations of the twists of $E$ as in Theorem A.

To obtain the Mordell-Weil rank of these twists, we will use the following result, which is an adaptation of a theorem obtained by Tate in \cite{Tate66} to our context. The application of Tate's result on the construction of elliptic curves with arbitrarily large rank appeared in \cite{TateSha67}, so it will be mentioned in this work as ``results from Tate and Shafarevich" or ``Tate-Shafarevich results''. For this notion, see also \cite[Chapter 13]{ShiodaSchuettBook}.

\begin{theorem}\label{tate-sha-results}
    Let $C$ be a curve over $\mathbb{F}_{3^{2n}}$. Denote by $L_E(T)$ the $L$-function of the elliptic curve $E: y^2 = x^3 - x$ over $\mathbb{F}_{3^{2n}}$. Write the $L$-function of $C$ over $\mathbb{F}_{3^{2n}}$ as
    $$L_C(T) = L_E(T)^{m(C)} \cdot G(T),$$
    with $\gcd(G, L_E) = 1$. Then
    $$\textrm{rank }E(\mathbb{F}_{3^{2n}}(C)) = 4m(C).$$
\end{theorem}

\section{The cubic twists}\label{cubic-twists}
In order to obtain equations for the cubic twists of $E$ in the function field of a curve $\mathcal{X}$, we proceed as in \cite[Section 3]{BGST}. We first consider $\Phi_{1, -1}$ as in Equation \eqref{morphisms-char3} and a curve $\mathcal{C}$ defined over $\mathbb{F}_{3^{2n}}$ that admits an order $3$ automorphism $\tau$ defined over $\mathbb{F}_{3^{2n}}$. Let $\mathcal{X} = \mathcal{C}/\langle \tau \rangle$. Since the curve $\mathcal{C}$ can be seen as a cyclic degree $3$ covering of $\mathcal{X}$, the function field of $\mathcal{C}$ can be seen as an Artin-Schreier extension of the function field of $\mathcal{X}$, that is, there exists $g \in \mathbb{F}_{3^{2n}}(\mathcal{X}) \backslash \mathbb{F}_{3^{2n}}$ such that $g \neq h^3 - h$ for any $h \in \mathbb{F}_{3^{2n}}(\mathcal{X})$ and $\mathbb{F}_{3^{2n}}(\mathcal{C}) = \mathbb{F}_{3^{2n}}(\mathcal{X})(w)$, where
$$w^3 - w = g$$
(see \cite[Proposition 3.7.8]{stichtenoth}). The automorphism $\tau$ maps $w$ into $w \pm 1$, and (considering $\tau^2$ instead of $\tau$, if necessary) we may assume $\tau(w) = w + 1$. 

In order to show that the morphisms
{\small$$(E \times \mathcal{C})/\langle \Phi_{1, -1} \times \tau \rangle \rightarrow \mathcal{C}/\langle \tau \rangle = \mathcal{X} \quad \textrm{and} \quad (E \times \mathcal{C})/\langle \Phi_{1, -1} \times \tau^2 \rangle \rightarrow \mathcal{C}/\langle \tau^2 \rangle = \mathcal{X}$$}
define elliptic surfaces over $\mathcal{X}$, we compute the fixed subfields $\mathbb{F}_{\tau}$ and $\mathbb{F}_{\tau^2}$ of $\mathbb{F}_{3^{2n}}(\mathcal{X})(x, y, w)$ under the action of these morphisms.

One has
$$
\Phi_{1, -1} \times \tau = \left\{ \begin{array}{ccl}
	x & \mapsto & x - 1 \\
	y & \mapsto & y \\
	w & \mapsto & w + 1 \\
\end{array} \right.,
$$
so it can be seen that
$$\mathbb{F}_{3^{2n}}(\mathcal{X})(y, x + w) \subseteq \mathbb{F}_{\tau} \subseteq \mathbb{F}_{3^{2n}}(\mathcal{X})(x, y, w).$$
Since 
$$[\mathbb{F}_{3^{2n}}(\mathcal{X})(x, y, w) : \mathbb{F}_{3^{2n}}(\mathcal{X})(y, x + w)] = 3 \quad \textrm{and} \quad x \notin \mathbb{F}_{\tau},$$
one obtains
$$\mathbb{F}_{\tau} = \mathbb{F}_{3^{2n}}(\mathcal{X})(y, x + w).$$

Let $\xi = x + w$, so
$$\xi^3 = x^3 + w^3 = y^2 + x + w + g = y^2 + \xi + g.$$
Therefore, the expression of the generic fiber of the elliptic surface given by $(E \times \mathcal{C})/\langle \Phi_{1, -1} \times \tau \rangle \rightarrow \mathcal{X}$ is
$$E_{3, \tau} : y^2 = \xi^3 - \xi - g.$$

The same analysis for $\Phi_{1, -1} \times \tau^2$ yields the following expression for the generic fiber of the elliptic surface:
$$E_{3, \tau^2}: y^2 = \xi^3 - \xi + g.$$

\begin{remark}
    Let $\Phi_{1, 1}$ be as in Equation \eqref{morphisms-char3}. One can check that repeating this procedure with the morphism $\Phi_{1, 1} \times \tau$ yields the elliptic fibration $E_{3, \tau^2}$, and the morphism $\Phi_{1, 1}\times \tau^2$ yields $E_{3, \tau}$.
\end{remark}

\begin{proposition}\label{prop:cubic-twists}
	The elliptic curve $E$ admits the cubic twists 
	$$E_{3, \tau} / \mathbb{F}_{3^{2n}}(\mathcal{X}): y^2 = \xi^3 - \xi - g \quad \textrm{and} \quad E_{3, \tau^2} / \mathbb{F}_{3^{2n}}(\mathcal{X}): y^2 = \xi^3 - \xi + g$$
	in the function field $\mathbb{F}_{3^{2n}}(\mathcal{C})$, which are isomorphic over $\mathbb{F}_{3^{2n}}(\mathcal{X})$.
\end{proposition}
\begin{proof}
	One can check that $E$ is isomorphic to $E_{3, \tau}$ and $E_{3, \tau^2}$ by considering the following morphisms:
	$$\begin{array}{ccc}
		E & \cong & E_{3, \tau} \\
		(x, y) & \mapsto & (x + w, y) \\
		(\xi - w, y) & \mapsfrom & (\xi, y)
	\end{array}
	\quad \textrm{and} \quad
	\begin{array}{ccc}
		E & \cong & E_{3, \tau^2} \\
		(x, y) & \mapsto & (x - w, y) \\
		(\xi + w, y) & \mapsfrom & (\xi, y)
	\end{array}.$$
	This isomorphism is defined over a cubic extension of $\mathbb{F}_{3^{2n}}(\mathcal{X})$, and cannot be done over any smaller extension (i.e., it cannot be done over $\mathbb{F}_{3^{2n}}(\mathcal{X})$). Therefore, we conclude that $E_{3, \tau}$ and $E_{3, \tau^2}$ are cubic twists of $E$.
	
	Furthermore, since $\mathbb{F}_{3^2} \subseteq \mathbb{F}_{3^{2n}}$, there is an element $i \in \mathbb{F}_{3^{2n}}$ such that $i^2 = -1$. The morphism $\nu: E_{3, \tau} \rightarrow E_{3, \tau^2}$ given by $\nu(\xi, y) = (-\xi, iy)$ admits the inverse morphism $\overline{\nu}: E_{3, \tau^2} \rightarrow E_{3, \tau}$ given by $\overline{\nu}(\xi, y) = (-\xi, -iy)$, and we conclude that $E_{3, \tau}$ and $E_{3, \tau^2}$ are isomorphic over $\mathbb{F}_{3^{2n}}(\mathcal{X})$.
\end{proof}

\subsection{The rank formula}
The previous proposition implies that the Mordell-Weil ranks of $E_{3, \tau}$ and of $E_{3, \tau^2}$ over $\mathbb{F}_{3^{2n}}(\mathcal{X})$ are the same. We will now prove one item of Theorem B: we will find a rank formula for these elliptic curves in terms of the decomposition of the $L$-functions of the curves $\mathcal{X}$ and $\mathcal{C}$. For this, we will decompose the set of $\mathbb{F}_{3^{2n}}(\mathcal{C})$-rational points of these twists into the sum of eigenspaces. This method was first used in \cite{BTT}.

\begin{theorem}\label{teo:rank-cubic-twist}
	Let $\mathcal{C}$ be a curve defined over $\mathbb{F}_{3^{2n}}$ that admits an order $3$ automorphism $\tau$, $\mathcal{X} = \mathcal{C}/ \langle \tau \rangle$ and $E_{3, \tau}$ and $E_{3, \tau^2}$ be the elliptic curves over $\mathbb{F}_{3^{2n}}(\mathcal{X})$ defined by
	$$E_{3, \tau}: y^2 = \xi^3 - \xi - g \quad \textrm{and} \quad E_{3, \tau^2}: y^2 = \xi^3 - \xi + g.$$
	Denote by $L_E(T)$ the $L$-function of the elliptic curve $E$ over $\mathbb{F}_{3^{2n}}$. Writing the $L$-functions of $\mathcal{X}$ and of $\mathcal{C}$ over $\mathbb{F}_{3^{2n}}$ as, respectively,
	$$L_{\mathcal{X}}(T) = L_{E}(T)^{m(\mathcal{X})} \cdot X(T) \quad \textrm{and} \quad L_{\mathcal{C}}(T) = L_{E}(T)^{m(\mathcal{C})} \cdot C(T),$$
	with $\gcd(L_E(T), X(T)) = \gcd(L_E(T), C(T)) = 1$, then
	$$\textrm{rank }E_{3, \tau}(\mathbb{F}_{3^{2n}}(\mathcal{X})) = \textrm{rank }E_{3, \tau^2}(\mathbb{F}_{3^{2n}}(\mathcal{X})) = 2m(\mathcal{C}) - 2m(\mathcal{X}).$$
\end{theorem}
\begin{proof}
	One can check that $\rho(\xi, y) = (\xi + 1, y)$ is an automorphism of $E_{3, \tau}$ satisfying $\rho^2 + \rho = [-1]$. Let $\zeta_3$ be a (complex) primitive third root of unity. Then,
	$$V = E_{3, \tau}(\mathbb{F}_{3^{2n}}(\mathcal{C})) \otimes \mathbb{Q}$$
	can be seen as a $\mathbb{Q}(\zeta_3)$-vector space, by defining
	$$\zeta_3 \cdot (P \otimes a) = \rho(P) \otimes a.$$
	The automorphism $\tau$ of $\mathbb{F}_{3^{2n}}(\mathcal{C})$ induces a $\mathbb{Q}(\zeta_3)$-linear map $\hat{\tau}$ on $V$ as follows: one can write $P \in \mathbb{F}_{3^{2n}}(\mathcal{C})$ as
	$$P = (x_0 + x_1w + x_2w^2, y_0 + y_1w + y_2w^2),$$
	where $x_i, y_j \in \mathbb{F}_{3^{2n}}(\mathcal{X})$. Then,
	$$\hat{\tau}(P \otimes a) = (x_0 + x_1 + x_2 + (x_1 - x_2)w + x_2w^2, y_0 + y_1 + y_2 + (y_1 - y_2)w + y_2w^2) \otimes a.$$
	The $\mathbb{Q}(\zeta_3)$ linearity of $\hat{\tau}$ follows from $\rho \circ \hat{\tau} = \hat{\tau} \circ \rho$.
	
	Since $\hat{\tau}^3 = \textrm{id}_V$, the vector space $V$ can be decomposed as the sum of $3$ eigenspaces, namely those associated to the third roots of unity $\{1, \zeta_3, \zeta_3^2\}$, that is,
	$$V = \textrm{Ker}(\hat{\tau} - \textrm{id}_V) \oplus \textrm{Ker}(\hat{\tau} - \zeta_3 \cdot \textrm{id}_V) \oplus \textrm{Ker}(\hat{\tau} - \zeta_3^2 \cdot \textrm{id}_V).$$
	We now study each eigenspace.
	\begin{itemize}
		\item $\textrm{Ker}(\hat{\tau} - \textrm{id}_V)$: Comparing coordinates, one sees that $0 \neq P \otimes a \in \textrm{Ker}(\hat{\tau} - \textrm{id}_V)$ if, and only if,
		$$P \otimes a = (x_0, y_0) \otimes a,$$
		with $(x_0, y_0) \in E_{3, \tau}(\mathbb{F}_{3^{2n}}(\mathcal{X}))$. Therefore, one concludes
		$$\textrm{Ker}(\hat{\tau} - \textrm{id}_V) = E_{3, \tau}(\mathbb{F}_{3^{2n}}(\mathcal{X})) \otimes \mathbb{Q},$$
		so
		$$\dim_{\mathbb{Q}} \textrm{Ker}(\hat{\tau} - \textrm{id}_V) = \textrm{rank }E_{3, \tau}(\mathbb{F}_{3^{2n}}(\mathcal{X})).$$
		
		\item $\textrm{Ker}(\hat{\tau} - \zeta_3 \cdot \textrm{id}_V)$: Now $0 \neq P \otimes a \in\textrm{Ker}(\hat{\tau} - \zeta_3 \cdot \textrm{id}_V)$ if, and only if,
		$$P \otimes a = (x_0 + w, y_0) \otimes a,$$
		with $(x_0, y_0) \in \mathbb{F}_{3^{2n}}(\mathcal{X})$ and $(x_0 + w, y_0) \in E_{3, \tau}$. Substituting this on the equation of $E_{3, \tau}$ allows one to conclude that this is equivalent to having $(x_0, y_0) \in E(\mathbb{F}_{3^{2n}}(\mathcal{X}))$ Therefore, one obtains
		$$\textrm{Ker}(\hat{\tau} - \zeta_3 \cdot \textrm{id}_V) = E(\mathbb{F}_{3^{2n}}(\mathcal{X})) \otimes \mathbb{Q},$$
		so
		$$\dim_{\mathbb{Q}} \textrm{Ker}(\hat{\tau} - \zeta_3 \cdot \textrm{id}_V) = \textrm{rank }E(\mathbb{F}_{3^{2n}}(\mathcal{X})).$$
		Then Theorem \ref{tate-sha-results} allows one to conclude that
		$$\textrm{rank }E(\mathbb{F}_{3^{2n}}(\mathcal{X})) = 4m(\mathcal{X}).$$
		
		\item $\textrm{Ker}(\hat{\tau} - \zeta_3^2 \cdot \textrm{id}_V)$: The final eigenspace is composed of elements $0 \neq P \otimes a$ such that $$P \otimes a = (x_0 - w, y_0) \otimes a,$$
		with $(x_0, y_0) \in \mathbb{F}_{3^{2n}}(\mathcal{X})$ and $(x_0 - w, y_0) \in E_{3, \tau}$. This is equivalent to having $(x_0, y_0) \in E_{3, \tau^2}(\mathbb{F}_{3^{2n}}(\mathcal{X}))$. Therefore, one obtains
		$$\textrm{Ker}(\hat{\tau} - \zeta_3 \cdot \textrm{id}_V) = E_{3, \tau^2}(\mathbb{F}_{3^{2n}}(\mathcal{X})) \otimes \mathbb{Q},$$
		so
		$$\dim_{\mathbb{Q}} \textrm{Ker}(\hat{\tau} - \zeta_3 \cdot \textrm{id}_V) = \textrm{rank }E_{3, \tau^2}(\mathbb{F}_{3^{2n}}(\mathcal{X})).$$
		We already showed that $\textrm{rank }E_{3, \tau}(\mathbb{F}_{3^{2n}}(\mathcal{X})) = \textrm{rank }E_{3, \tau^2}(\mathbb{F}_{3^{2n}}(\mathcal{X}))$.
	\end{itemize}
	
	Since $E_{3, \tau}$ is a twist of $E$ over $\mathbb{F}_{3^{2n}}(\mathcal{C})$, it holds that $E \cong E_{3, \tau}$ over this function field and, in particular, $\textrm{rank }E_{3, \tau}(\mathbb{F}_{3^{2n}}(\mathcal{C})) = \textrm{rank }E(\mathbb{F}_{3^{2n}}(\mathcal{C}))$. By using again Theorem \ref{tate-sha-results}, one concludes that $\dim_{\mathbb{Q}}(V) = \textrm{rank }E_{3, \tau}(\mathbb{F}_{3^{2n}}(\mathcal{C})) = 4m(\mathcal{C})$, Therefore,
	$$4m(\mathcal{C}) = 2 \cdot \textrm{rank }E_{3, \tau}(\mathbb{F}_{3^{2n}}(\mathcal{X})) + 4m(\mathcal{X}).$$
\end{proof}

\begin{corollary}\label{maximal-cubic}
	Let $\mathcal{X}$ be a curve defined over $\mathbb{F}_{3^{2n}}$ and $\mathcal{C}$ be an Artin-Schreier extension of $\mathcal{X}$ defined by
	$$\mathcal{C}: w^3 - w = g$$
	which is maximal (resp. minimal) over $\mathbb{F}_{3^{2n}}$ for odd (resp. even) $n$. If $g(\mathcal{X})$ and $g(\mathcal{C})$ denote respectively the genera of $\mathcal{X}$ and $\mathcal{C}$, then
	$$\textrm{rank }E_{3, \tau}(\mathbb{F}_{3^{2n}}(\mathcal{X})) = \textrm{rank }E_{3, \tau^2}(\mathbb{F}_{3^{2n}}(\mathcal{X})) = 2g(\mathcal{C}) - 2g(\mathcal{X}).$$
\end{corollary}
\begin{proof}
    This corollary follows from the fact that, if $\mathcal{C}$ is a maximal curve, then $\mathcal{X}$ is also maximal (from Serre's theorem for maximal curves), and then the $L$-function of such curves over $\mathbb{F}_{3^{2n}}$ are given by
	{\small $$L_{\mathcal{X}}(T) = (3^nT+1)^{2g(\mathcal{X})} = L_E(T)^{g(\mathcal{X})} \quad \textrm{and} \quad L_{\mathcal{C}}(T) = (3^nT + 1)^{2g(\mathcal{C})} = L_E(T)^{g(\mathcal{C})}.$$}
    (This is a consequence of \cite[Corollary 5.1.17]{stichtenoth}). In the minimal case, the $L$-functions of $\mathcal{X}$ and $\mathcal{C}$ are given by
	$$L_{\mathcal{X}}(T) = (3^nT - 1)^{2g(\mathcal{X})} \quad \textrm{and} \quad L_{\mathcal{C}}(T) = (3^nT - 1)^{2g(\mathcal{C})}.$$
\end{proof}



The following example proves Theorem C. As Tate and Shafarevich, we were able to obtain elliptic curves having arbitrarily large rank over a function field of characteristic $3$, but now using cubic twists instead of quadratic ones.

\begin{example}\label{bouw-example}
	Let $n$ be odd and $\alpha \in \mathbb{F}_{3^{2n}}$ be such that $\alpha^{3^n-1} = -1$. Then, the elliptic curves over $\mathbb{F}_{3^{2n}}(t)$
	$$E_{3, \tau}: y^2 = \xi^3 - \xi - \alpha t^{3^n+1} \quad \textrm{and} \quad E_{3, \tau^2}: y^2 = \xi^3 - \xi + \alpha t^{3^n+1}$$
	are such that 
	$$\textrm{rank }E_{3, \tau}(\mathbb{F}_{3^{2n}}(t)) = \textrm{rank }E_{3, \tau^2}(\mathbb{F}_{3^{2n}}(t)) = 2 \cdot 3^n.$$
	To prove this, consider the Artin-Schreier curve
	$$\mathcal{C}: y^3 - y = \alpha t^{3^n+1}.$$
	From \cite[Proposition 1.9.3]{bouw}, one concludes that $\mathcal{C}$ is a maximal curve over $\mathbb{F}_{3^{2n}}$. The result now follows from Corollary \ref{maximal-cubic} by observing that, in this case, $\mathcal{X} = \mathbb{P}^1$ has genus $0$. 
\end{example}

\subsection{The singular fiber}
We now consider the case that $g$ is a polynomial in $t$ (that is, $\mathcal{X} = \mathbb{P}^1$). We will show that the elliptic fibrations defined by the cubic twists we obtained have only one singular fiber and we classify it in terms of the congruence of the degree of $g$ modulo $6$ (see Proposition \ref{fiber-cubic-twist}). Finally, we apply Shioda-Tate formula to obtain the rank of those cubic twists over $\overline{\mathbb{F}}_{3^{2n}}(t)$.

In the case that $g$ is a polynomial, we can use a change of variables $\xi \mapsto \xi + c$ for some $c \in \mathbb{F}_{3^{2n}}[t]$, and it is possible to assume that $3 \nmid \deg g(t)$. The following result states that the elliptic fibrations defined by $E_{3, \tau}$ and $E_{3, \tau^2}$ have only one singular fiber, and its type only depends on the degree of $g$.

\begin{proposition}\label{fiber-cubic-twist}
	Let $g \in \mathbb{F}_{3^{2n}}[t]$ be a polynomial such that $3 \nmid \deg g$. Consider the elliptic fibrations defined by $E_{3, \tau}$ and $E_{3, \tau^2}$ as in Proposition \ref{prop:cubic-twists}. The only singular fiber of these fibrations is at infinity, and its type (following Kodaira's classification) can be found on the following table.
	
	\begin{center}
		\begin{tabular}{c|c}
			Condition on the degree of $g$ & Fiber type \\
			\hline
			$\deg g \equiv 1$ (mod $6$) & $II^*$ \\
			$\deg g \equiv 2$ (mod $6$) & $IV^*$ \\
			$\deg g \equiv 4$ (mod $6$) & $IV$ \\
			$\deg g \equiv 5$ (mod $6$) & $II$
		\end{tabular}
	\end{center}
\end{proposition}
\begin{proof}
	Consider the elliptic fibrations defined by
	$$E_{3, \tau}: y^2 = \xi^3 - \xi - g(t) \quad \textrm{and} \quad E_{3, \tau^2}: y^2 = \xi^3 - \xi + g(t).$$
	The discriminant of both equations is $\Delta \equiv 1$. Therefore, the only possible singular fibers of such fibrations is at the poles of $g$. Since $g$ is a polynomial, this means that the only possible singular fiber is at infinity.
	
	To study its behavior, one can follow Tate's algorithm. For this, define $n$ to be the Euler number of the elliptic fibrations, that is, the smallest integer such that
	$$\deg g(t) \leq 6n \Longleftrightarrow n = \left\lceil \frac{\deg g}{6} \right\rceil.$$
	Defining $\gamma = \frac{\xi}{t^{2n}}$, $\delta = \frac{y}{t^{3n}}$ and $s = \frac{1}{t}$ and dividing the equations of $E_{3, \tau}$ and $E_{3, \tau^2}$ by $t^{6n}$ yields
	$$\delta^2 = \gamma^3 - s^{4n}\gamma - s^{6n - \deg g}\tilde{g}(s) \quad \textrm{and} \quad \delta^2 = \gamma^3 - s^{4n}\gamma + s^{6n - \deg g}\tilde{g}(s),$$
	where $\tilde{g}$ is a polynomial in $s$. A direct application of Tate's algorithm (see \cite{tate}) now allows one to study the singular fiber at $s = 0$.
\end{proof}

\begin{remark}\label{irreducible-fibers}
	This proposition shows that the elliptic surfaces $E_{3, \tau}$ and $E_{3, \tau^2}$ have an interesting property, which is particular to fields of characteristic $2$ and $3$ (see \cite{WLang}): the presence of a unique singular fiber in an elliptic fibration. 
	
	Moreover, this proposition shows that, if $\deg g \equiv 5$ (mod $6$), then all fibers of the elliptic fibrations are irreducible, and only one of them is singular. If one can construct families of maximal Artin-Schreier curves $\mathcal{C}$ with defining polynomial having degree $5$ modulo $6$ and arbitrarily large genus, then one would obtain examples of elliptic fibrations having only irreducible fibers and the generic fiber having arbitrarily large Mordell-Weil rank.
\end{remark}

\begin{theorem}\label{geometric-cubic}
	Let $r$ be the geometric Mordell-Weil rank of $E_{3, \tau}$, which is known to be the same as that of $E_{3, \tau^2}$. Then
	$$r = 2 \deg g - 2.$$
\end{theorem}
\begin{proof}
	It follows from the proof of the previous proposition that the second Betti number of this elliptic surfaces is
	$$b_2(E_{3, \tau}) = 12 \left\lceil \frac{\deg g}{6} \right\rceil - 2.$$
	(see \cite[Section 6.10]{ShiodaSchuett}).
	
	Since $E_{3, \tau}$ is a supersingular elliptic curve, it follows that $\rho(E_{3, \tau})$, the rank of the Néron-Severi groups of the elliptic surface associated to it, equals its second Betti number. The result now follows from Shioda-Tate formula (see \cite[Corollary 6.13]{ShiodaSchuett}) and the information on its singular fiber.
\end{proof}

As a corollary of this theorem, we obtain the following result.

\begin{corollary}
	If $\mathcal{C}$ is an Artin-Schreier curve which is maximal (respectively, minimal) over $\mathbb{F}_{3^{2n}}$ for odd (respectively, even) $n$, then the geometric rank Mordell-Weil rank of $E_{3, \tau}$ and $E_{3, \tau^2}$ is attained over $\mathbb{F}_{3^{2n}}(t)$.
\end{corollary}
\begin{proof}
	The genus of an Artin-Schreier curve given by 
	$$\mathcal{C}: w^3 - w = g(t)$$
	is $g(\mathcal{C}) = \deg g - 1$ (see \cite[Proposition 3.7.8]{stichtenoth}). This corollary is now a combination of Theorem \ref{geometric-cubic} and Corollary \ref{maximal-cubic}.
\end{proof}

\section{The other twists}\label{other-twists}
We now discuss the other twists of $E: y^2 = x^3 - x$, namely the quadratic, quartic and sextic ones. Many results here are similar to those of the previous section, hence we will only present minor details regarding the differences between these cases and the cubic twist one.

\vspace{0.3cm}

\noindent{\large \textbf{The quadratic twist}}

\vspace{0.3cm}

\noindent Consider $\Phi_{-1, 0}$ as in Equation \eqref{morphisms-char3} and a curve $\mathcal{D}$ that admits an order $2$ morphism $\tau$. Let $\mathcal{X} = \mathcal{D}/\langle \tau \rangle$. Since $\mathcal{D}$ can be seen as a cyclic degree $2$ covering of $\mathcal{X}$, there exists $f \in \mathbb{F}_{3^{2n}}(\mathcal{X})$ such that $f$ is not a square and $\mathbb{F}_{3^{2n}}(\mathcal{D}) = \mathbb{F}_{3^{2n}}(\mathcal{X})(s)$, where
$$s^2 = f \quad \textrm{and} \quad \tau(s) = -s$$
(see \cite[Proposition 3.7.3]{stichtenoth}). 

\begin{proposition}\label{prop:quadratic-twists}
	Let $\mathcal{X}$ be a curve defined over $\mathbb{F}_{3^{2n}}$ and let $f \in \mathbb{F}_{3^{2n}}(\mathcal{X})$ be a non-square element. The elliptic curve $E / \mathbb{F}_3: y^2 = x^3 - x$ admits the quadratic twist $E_2: \eta^2 = \xi^3 - f^2\xi$ over the function field $\mathbb{F}_{3^{2n}}(\mathcal{X})$.
\end{proposition}
\begin{proof}
	Analogously to the cubic twist cases, we will first compute the fixed subfield $\mathbb{F}$ of $\mathbb{F}_{3^{2n}}(\mathcal{X})(x, y, s)$ under the action of the morphism $\Phi_{-1, 0} \times \tau$. One can check that
	$$\mathbb{F}_{3^{2n}}(\mathcal{X})(s^2x, s^3y) \subseteq \mathbb{F} \subseteq \mathbb{F}_{3^{2n}}(\mathcal{X})(x, y, s).$$
	Since the minimal polynomial of $s$ over $\mathbb{F}_{3^{2n}}(\mathcal{X})(s^2x, s^3y)$ is $T^2 - f$, we conclude that the field extension $\mathbb{F}_{3^{2n}}(\mathcal{X})(x, y, s) / \mathbb{F}_{3^{2n}}(\mathcal{X})(s^2x, s^3y)$ has degree $2$ and, since $s \not\in \mathbb{F}$, we obtain
	$$\mathbb{F} = \mathbb{F}_{3^{2n}}(\mathcal{X})(s^2x, s^3y).$$
	
	Denoting $\xi = s^2x$ and $\eta = s^3y$, the expression of the generic fiber of the elliptic surface given by $(E \times \mathcal{D})/\langle \Phi_{-1,0} \times \tau \rangle \rightarrow \mathcal{X}$ is
	$$E_2: \eta^2 = \xi^3 - f^2\xi.$$
	
	The isomorphism between $E$ and $E_2$ is given by
	$$\begin{array}{ccc}
		E & \cong & E_2 \\
		(x, y) & \mapsto & (s^2x, s^3y) \\
		(s^{-2}\xi, s^{-3}\eta) & \mapsfrom & (\xi, \eta)
	\end{array},$$
	which is defined over $\mathbb{F}_{3^{2n}}(\mathcal{X})(s)$, but not over $\mathbb{F}_{3^{2n}}(\mathcal{X})$, and hence $E_2$ defines a quadratic twist of $E$ in this function field.
\end{proof}

\begin{proposition}\label{rank-quadratic-twist}
	Let $\mathcal{D}$ be a curve defined over $\mathbb{F}_{3^{2n}}$ that admits an order $2$ automorphism $\tau$, $\mathcal{X} = \mathcal{D}/ \langle \tau \rangle$ and $E_2$ be the elliptic curve over $\mathbb{F}_{3^{2n}}(\mathcal{X})$ defined by
	$$E_2: \eta^2 = \xi^3 - f^2\xi.$$
	Denote by $L_{E}(T)$ the $L$-function of the elliptic curve $E$ over $\mathbb{F}_{3^{2n}}$. Writing the $L$-functions of $\mathcal{X}$ and of $\mathcal{D}$ as, respectively,
	$$L_{\mathcal{X}}(T) = L_{E}(T)^{m(\mathcal{X})} \cdot X(T) \quad \textrm{and} \quad L_{\mathcal{D}}(T) = L_{E}(T)^{m(\mathcal{D})} \cdot D(T),$$
	with $\gcd(L_{E}(T), X(T)) = \gcd(L_{E}(T), D(T)) = 1$, then
	$$\textrm{rank}\ E_2(\mathbb{F}_{3^{2n}}(\mathcal{X})) = 4[m(\mathcal{D}) - m(\mathcal{X})].$$
\end{proposition}
\begin{proof}
	The proof is similar to the cubic twist case. The morphism $\hat{\tau}$ on $V = E_2(\mathbb{F}_{3^{2n}}(\mathcal{D})) \otimes \mathbb{Q}$ induced by $\tau$ is now given by
	$$\hat{\tau}((x_0 + x_1s, y_0 + y_1s) \otimes a) = (x_0 - x_1s, y_0 - y_1s) \otimes a,$$
	where $x_0, x_1, y_0, y_1 \in \mathbb{F}_{3^{2n}}(\mathcal{X})$.
	The space $V$ will be divided into two eigenspaces:
	
	\begin{itemize}
		\item $\textrm{Ker}(\hat{\tau} - \textrm{id}_V)$: Comparing coordinates, we conclude that $0 \neq P \otimes a \in \textrm{Ker}(\hat{\tau} - \textrm{id}_V)$ if, and only if,
		$$P \otimes a = (x_0, y_0) \otimes a,$$
		with $(x_0, y_0) \in E_2(\mathbb{F}_{3^{2n}}(\mathcal{X}))$. Therefore, we obtain
		$$\textrm{Ker}(\hat{\tau} - \textrm{id}_V) = E_2(\mathbb{F}_{3^{2n}}(\mathcal{X})) \otimes \mathbb{Q},$$
		so
		$$\dim_{\mathbb{Q}}\textrm{Ker}(\hat{\tau} - \textrm{id}_V) = \textrm{rank}\ E_2(\mathbb{F}_{3^{2n}}(\mathcal{X})).$$
		
		\item $\textrm{Ker}(\hat{\tau} + \textrm{id}_V)$: Given $P = (x, y) \in E_2$, we have $-P = (x, -y)$ (see \cite[Group Law Algortihm 2.3]{silverman}), therefore, we have $0 \neq P \otimes a \in \textrm{Ker}(\hat{\tau} + \textrm{id}_V)$ if, and only if,
		$$P \otimes a = (x_0, y_1s) \otimes a,$$
		with $(x_0, y_1s) \in E_2(\mathbb{F}_{3^{2n}}(\mathcal{X})(s))$ and $x_0, y_1 \in \mathbb{F}_{3^{2n}}(\mathcal{X})$. This is equivalent to having $(f^{-1}x_0, f^{-1}y_1) \in E(\mathbb{F}_{3^{2n}}(\mathcal{X}))$. Theorem \ref{tate-sha-results} allows us to conclude that
		$$\dim_{\mathbb{Q}}\textrm{Ker}(\hat{\tau} + \textrm{id}_V) = 4m(\mathcal{X}).$$
	\end{itemize}
	
	Finally, since $E$ and $E_2$ are isomorphic over $\mathbb{F}_{3^{2n}}(\mathcal{D})$ and by using Theorem \ref{tate-sha-results}, we conclude that
	$$\begin{array}{rl}
	4m(\mathcal{D}) &= \textrm{rank}\ E(\mathbb{F}_{3^{2n}}(\mathcal{D})) = \textrm{rank}\ E_2(\mathbb{F}_{3^{2n}}(\mathcal{D})) \\ &= \textrm{rank}\ E_2(\mathbb{F}_{3^{2n}}(\mathcal{X})) + 4m(\mathcal{X}).
	\end{array}$$
\end{proof}

\begin{remark}
    The previous proposition extends the groundbreaking construction of Tate and Shafarevich in \cite{TateSha67}, since they considered $\mathcal{X} = \mathbb{P}^1$ (hence the second eigenspace has dimension $0$) and $\mathcal{D}: s^2 = \gamma x^f + \delta$, for $\gamma, \delta \in \mathbb{F}_3^*$ and $f \mid (3^{2m + 1} + 1)$ for some $m$.
\end{remark}

\begin{corollary}
	Let $\mathcal{X}$ be a curve defined over $\mathbb{F}_{3^{2n}}$ and $\mathcal{D}$ be a Kummer extension of $\mathcal{X}$ given by $s^2 = f$, for some non-square element $f \in \mathbb{F}_{3^{2n}}(\mathcal{X})$. If $\mathcal{D}$ is maximal (resp. minimal) over $\mathbb{F}_{3^{2n}}$ for odd (resp. even) $n$ and $g(\mathcal{D})$ and $g(\mathcal{X})$ denote, respectively, the genera of the curves $\mathcal{D}$ and $\mathcal{X}$, then
	$$\textrm{rank }E_2(\mathbb{F}_{3^{2n}}(\mathcal{X})) = 4g(\mathcal{D}) - 4g(\mathcal{X}).$$
\end{corollary}

In the special case that $\mathcal{X} = \mathbb{P}^1$, we can study the fibers of the elliptic fibration induced by $E_2$.

\begin{proposition}\label{fiber-quadratic}
	Let $f \in \mathbb{F}_{3^{2n}}[t]$ be a non-square polynomial. Consider the elliptic surface $E_2$ given by
	$$E_2: \eta^2 = \xi^3 - f(t)^2\xi.$$
	Then
	\begin{itemize}
		\item The fiber at $t = \alpha$, $\alpha \in \overline{\mathbb{F}}_{3^{2n}}$, is of type $I_0$ if the multiplicity of $\alpha$ as a root of $f$ is even or $0$;
		\item The fiber at $t = \alpha$, $\alpha \in \overline{\mathbb{F}}_{3^{2n}}$, is of type $I_0^*$ if the multiplicity of $\alpha$ as a root of $f$ is odd;
		\item The fiber at infinity is of type $I_0$ if $\deg f$ is even;
		\item The fiber at infinity is of type $I_0^*$ if $\deg f$ is odd.
	\end{itemize}
\end{proposition}
\begin{proof}
	We are considering that $E_2$ defines an elliptic fibration over $\mathbb{P}^1(\overline{\mathbb{F}}_{3^{2n}})$. We then need to study the fibers associated to the prime ideals of $\overline{\mathbb{F}}_{3^{2n}}(t)$, that is, $\pi = (t - \alpha)$, for $\alpha \in \overline{\mathbb{F}}_{3^{2n}}$, or $\pi = \frac{1}{t}$. 
	
	If we fix $\alpha \in \overline{\mathbb{F}}_{3^{2n}}$ (and so $\pi = (t - \alpha)$) and we write $f(t) = \pi^{2m + k}\tilde{f}(t)$, with $k \in \{0, 1\}$ and $\pi \nmid \tilde{f}(t)$, we can change the defining equation of $E_2$ into the minimal Weierstrass equation
	$$y^2 = x^3 - \pi^{2k}\tilde{f}^2x$$
	by defining $\eta = \pi^{3m}y$ and $\xi = \pi^{2m}x$. In order to study the fiber associated to $\pi = \frac{1}{t}$, we define $s = \frac{1}{t}$, $n = \lceil \frac{\deg f}{2} \rceil$, $y = \eta t^{-3n}$ and $x = \xi t^{-2n}$. By dividing the defining equation of $E_2$ by $t^{6n}$, we obtain
	$$y^2 = x^3 - s^{4n - 2\deg f} \tilde{f}(s)x.$$
	
	The result now follows by a direct application of Tate's algorithm.
\end{proof}

\vspace{0.3cm}

\noindent{\large \textbf{The quartic twists}}

\vspace{0.3cm}

\noindent Let $i \in \mathbb{F}_9$ be such that $i^2 = -1$ and $t \in \mathbb{F}_3$, and consider $\Phi_{i, t}$ as in Equation \eqref{morphisms-char3} and a curve $\mathcal{G}$ defined over $\mathbb{F}_{3^{2n}}$ admitting an order $4$ automorphism $\tau$. Let $\mathcal{X} = \mathcal{G} / \langle \tau \rangle$. The curve $\mathcal{G}$ can be seen as a cyclic degree $4$ covering of $\mathcal{X}$. Since $3 \nmid 4$, there exists $f \in \mathbb{F}_{3^{2n}}(\mathcal{X})$ such that $f$ is not a square in $\mathbb{F}_{3^{2n}}(\mathcal{X})$ (in particular, $f$ is not a fourth-power of any element) and $\mathbb{F}_{3^{2n}}(\mathcal{G}) = \mathbb{F}_{3^{2n}}(\mathcal{X})(r)$, with
$$r^4 = f \quad \textrm{and} \quad \tau(r) = -ir.$$

\begin{proposition}\label{prop:quartic-twists}
	Let $\mathcal{X}$ be a curve defined over $\mathbb{F}_{3^{2n}}$ and let $f \in \mathbb{F}_{3^{2n}}(\mathcal{X})$ be such that $f$ is not a square in $\mathbb{F}_{3^{2n}}(\mathcal{X})$. The elliptic curve $E / \mathbb{F}_3: y^2 = x^3 - x$ admits the quartic twists 
	$$E_{4, \tau}: \eta^2 = \xi^3 - f\xi \quad \textrm{and} \quad E_{4, \tau^3}: \eta^2 = \xi^3 - f^3\xi$$ 
	over the function field $\mathbb{F}_{3^{2n}}(\mathcal{X})$.
\end{proposition}
\begin{proof}
	We have
	$$\Phi_{i, t} \times \tau: (x, y, r) \mapsto (-x + t, -iy, -ir).$$
	Denoting by $\mathbb{F}$ the fixed subfield of $\mathbb{F}_{3^{2n}}(\mathcal{X})(x, y, r)$ by $\Phi_{i, t} \times \tau$, we have
	$$\mathbb{F}_{3^{2n}}(\mathcal{X})(r^2(x + t), r^3y) \subseteq \mathbb{F}.$$
	One can check that equality holds by noting that the field extension
	$$\mathbb{F}_{3^{2n}}(\mathcal{X})(x, y, r) / \mathbb{F}_{3^{2n}}(\mathcal{X})(r^2(x + t), r^3y)$$
	has basis $\{1, r, r^2, r^3\}$ and compute the image of these elements under $\Phi_{i, t} \times \tau$. By denoting $\xi = r^2(x + t)$ and $\eta = r^3y$, we conclude that the generic fiber of the elliptic surface $(E \times \mathcal{C}) / \langle \Phi_{i, t} \times \tau \rangle \rightarrow \mathcal{X}$ is given by
	$$E_{4, \tau}: \eta^2 = \xi^3 - f\xi.$$
	
	An analogous procedure with the morphisms $\Phi_{i, t} \times \tau^3$, $\Phi_{-i, t} \times \tau$ and $\Phi_{-i, t} \times \tau^3$, which also have order $4$, yield, respectively, the elliptic curves
	$$E_{4, \tau^3}: \eta^2 = \xi^3 - f^3\xi,$$
	$E_{4, \tau^3}$ and $E_{4, \tau}$.
	
	The isomorphisms between $E$ and these elliptic curves are given by
	$$\begin{array}{ccc}
		E & \cong & E_{4, \tau} \\
		(x, y) & \mapsto & (r^2(x + t), r^3y) \\
		(r^{-2}\xi - t, r^{-3}\eta) & \mapsfrom & (\xi, \eta)
	\end{array}$$
	and
	$$\begin{array}{ccc}
		E & \cong & E_{4, \tau^3} \\
		(x, y) & \mapsto & (r^6(x + t), r^9y) \\
		(r^{-6}\xi - t, r^{-9}\eta) & \mapsfrom & (\xi, \eta)
	\end{array},
	$$
	and they are defined over a degree $4$ extension of $\mathbb{F}_{3^{2n}}(\mathcal{X})$.
\end{proof}

\begin{proposition}
	Let $\mathcal{G}$ be a curve defined over $\mathbb{F}_{3^{2n}}$ that admits an order $4$ automorphism $\tau$, $\mathcal{X} = \mathcal{G}/ \langle \tau \rangle$, $\mathcal{D} = \mathcal{G}/ \langle \tau^2 \rangle$ and $E_{4, \tau}$ and $E_{4, \tau^3}$ be the elliptic curves over $\mathbb{F}_{3^{2n}}(\mathcal{X})$ defined by
	$$E_{4, \tau}: \eta^2 = \xi^3 - f\xi \quad \textrm{and} \quad  E_{4, \tau^3}: \eta^2 = \xi^3 - f^3\xi.$$
	Denote by $L_{E}(T)$ the $L$-function of the elliptic curve $E$ over $\mathbb{F}_{3^{2n}}$ and write the $L$-functions of $\mathcal{X}$, of $\mathcal{D}$ and of $\mathcal{G}$ as, respectively,
	$$L_{\mathcal{X}}(T) = L_{E}(T)^{m(\mathcal{X})} \cdot X(T), \quad L_{\mathcal{D}}(T) = L_{E}(T)^{m(\mathcal{D})} \cdot D(T)$$
	and
	$$L_{\mathcal{G}}(T) = L_{E}(T)^{m(\mathcal{G})} \cdot G(T),$$
	with 
	$$\gcd(L_{E}(T), X(T)) = \gcd(L_{E}(T), D(T)) = \gcd(L_{E}(T), G(T)) = 1,$$
	then 
	$$\textrm{rank}\ E_{4, \tau}(\mathbb{F}_{3^{2n}}(\mathcal{X})) = \textrm{rank}\ E_{4, \tau^3}(\mathbb{F}_{3^{2n}}(\mathcal{X})) = 2[m(\mathcal{G}) - m(\mathcal{D})].$$
\end{proposition}
\begin{proof}
	The morphism $\Phi_{i,0}: (\xi, \eta) \mapsto (-\xi, -i\eta)$ can be seen as an order $4$ morphism of $E_{4, \tau}$. Since $\Phi_{i, 0} \circ \tau = \tau \circ \Phi_{i, 0}$, we can see $V = E_{4,\tau}(\mathbb{F}_{3^{2n}}(\mathcal{G})) \otimes \mathbb{Q}$ as an $\mathbb{Q}(i)$-vector space, by defining
	$$i \cdot (P \otimes a) = \Phi_{i, 0}(P) \otimes a.$$
	The space $V$ can be divided into $4$ eigenspaces by considering the morphism $\hat{\tau}$ induced by $\tau$.
	
	\begin{itemize}
		\item $\textrm{Ker}(\hat{\tau} - \textrm{id}_V)$: Comparing coordinates, one sees that $0 \neq P \otimes a \in \textrm{Ker}(\hat{\tau} - \textrm{id}_V)$ if, and only if,
		$$P \otimes a = (x, y) \otimes a,$$
		with $(x, y) \in E_{4, \tau}(\mathbb{F}_{3^{2n}}(\mathcal{X}))$. Therefore, one obtains
		$$\textrm{Ker}(\hat{\tau} - \textrm{id}_V) = E_{4, \tau}(\mathbb{F}_{3^{2n}}(\mathcal{X})) \otimes \mathbb{Q},$$
		so
		$$\dim_{\mathbb{Q}} \textrm{Ker}(\hat{\tau} - \textrm{id}_V) = \textrm{rank }E_{4, \tau}(\mathbb{F}_{3^{2n}}(\mathcal{X})).$$
		
		\item $\textrm{Ker}(\hat{\tau} + \textrm{id}_V)$: In this case, $P \otimes a$ is in this kernel if, and only if, $P = (x, yr^2)$, with $(x, y) \in E_{4, \tau}(\mathbb{F}_{3^{2n}}(\mathcal{X}))$. Substituting in the equation of $E_{4, \tau}$ shows that this is equivalent to having $(fx, f^2y) \in E_{4, \tau^3}(\mathbb{F}_{3^{2n}}(\mathcal{X}))$. Therefore,
		$$\dim_{\mathbb{Q}} \textrm{Ker}(\hat{\tau} + \textrm{id}_V) = \textrm{rank }E_{4, \tau^3}(\mathbb{F}_{3^{2n}}(\mathcal{X})).$$
		Since the characteristic of $\mathbb{F}_{3^{2n}}(\mathcal{X})$ is $3$, we have a rational isogeny between $E_{4, \tau}$ and $E_{4, \tau^3}$ given by the Frobenius morphism $(x, y) \mapsto (x^3, y^3)$. This implies that $\textrm{rank }E_{4, \tau}(\mathbb{F}_{3^{2n}}(\mathcal{X}))$ and $\textrm{rank }E_{4, \tau^3}(\mathbb{F}_{3^{2n}}(\mathcal{X}))$ are the same, so 
		$$\dim_{\mathbb{Q}} \textrm{Ker}(\hat{\tau} - \textrm{id}_V) = \dim_{\mathbb{Q}} \textrm{Ker}(\hat{\tau} + \textrm{id}_V).$$
		
		\item $\textrm{Ker}(\hat{\tau} - i \cdot \textrm{id}_V)$: We now conclude
		$$P \otimes a \in \textrm{Ker}(\hat{\tau} - i \cdot \textrm{id}_V) \Longleftrightarrow P \otimes a = (xr^2,yr) \otimes a,$$
		with $(x, y) \in E_{4, \tau}(\mathbb{F}_{3^{2n}}(\mathcal{X}))$. This is equivalent to having $(fx, fy) \in E_{\mathcal{D}}$, the elliptic curve given by
		$$E_{\mathcal{D}}: \eta = \xi^3 - f^2\xi.$$
		Let $s = r^2$. Then the curve $\mathcal{D}$ whose function field is given by $\mathbb{F}_{3^{2n}}(\mathcal{X})(s)$ (with $s^2 = f$) can be seen as a degree $2$ extension of $\mathcal{X}$, and so $E_{\mathcal{D}}$ is a quadratic twist of $E$ (indeed, this is the same elliptic curve obtained in Proposition \ref{prop:quadratic-twists}). Therefore, from Proposition \ref{rank-quadratic-twist},
		$$\dim_{\mathbb{Q}} \textrm{Ker}(\hat{\tau} - i \cdot \textrm{id}_V) = \textrm{rank }E_{\mathcal{D}}(\mathbb{F}_{3^{2n}}(\mathcal{X})) = 4m(\mathcal{D}) - 4m(\mathcal{X}).$$
		
		\item $\textrm{Ker}(\hat{\tau} + i \cdot \textrm{id}_V)$: Finally, we have
		$$P \otimes a \in \textrm{Ker}(\hat{\tau} + i \cdot  \textrm{id}_V) \Longleftrightarrow P \otimes a = (xr^2, yr^3) \otimes a,$$
		with $(x, y) \in E_{4, \tau}(\mathbb{F}_{3^{2n}}(\mathcal{X}))$. This is equivalent to having $(x, y) \in E$. Therefore,
		$$\dim_{\mathbb{Q}} \textrm{Ker}(\hat{\tau} + i \cdot \textrm{id}_V) = \textrm{rank }E(\mathbb{F}_{3^{2n}}(\mathcal{X})) = 4m(\mathcal{X}).$$
	\end{itemize}
	
	The result now follows by observing that, as before,  $\dim V = 4m(\mathcal{G})$.
\end{proof}

\begin{remark}
    The previous proposition is an extension of the results in \cite{BTT}, since the authors considered $\mathcal{X} = \mathbb{P}^1$ and $\mathcal{G}$ a maximal curve.
\end{remark}

\begin{corollary}
	Let $\mathcal{X}$ be a curve defined over $\mathbb{F}_{3^{2n}}$ and $\mathcal{G}$ be a Kummer extension of $\mathcal{X}$ given by $r^4 = f$, for some non-square element $f \in \mathbb{F}_{3^{2n}}(\mathcal{X})$. Let $\mathcal{D}$ be a Kummer extension of $\mathcal{X}$ given by $s^2 = f$. If $\mathcal{G}$ is maximal (resp. minimal) over $\mathbb{F}_{3^{2n}}$ for odd (resp. even) $n$ and $g(\mathcal{G})$ and $g(\mathcal{D})$ denote, respectively, the genera of the curves $\mathcal{G}$ and $\mathcal{D}$, then
	$$\textrm{rank}\ E_{4, \tau}(\mathbb{F}_{3^{2n}}(\mathcal{X})) = \textrm{rank}\ E_{4, \tau^3}(\mathbb{F}_{3^{2n}}(\mathcal{X})) = 2g(\mathcal{G}) - 2g(\mathcal{D}).$$
\end{corollary}

The proofs of the following propositions are a direct application of Tate's algorithm, after considering some changes that are analogous to those we used in Proposition \ref{fiber-quadratic}, and so they are omitted.

\begin{proposition}
	Let $f \in \mathbb{F}_{3^{2n}}[t]$ be a non-square polynomial. Consider the elliptic surface $E_{4, \tau}$ given by
	$$E_{4, \tau}: \eta^2 = \xi^3 - f\xi.$$
	For $\alpha \in \overline{\mathbb{F}}_{3^{2n}}$, write $f(t) = (t - \alpha)^{4m + k}\tilde{f}(t)$, where $\tilde{f}(\alpha) \neq 0$ and $k \in \{0, 1, 2, 3\}$. For the fiber at infinity, we define
	$$\left\lceil \frac{\deg f}{4} \right\rceil = \frac{\deg f + k}{4}.$$
	The fibers at $\alpha$ and at $\infty$ are given in the following table.
	
	\begin{center}
		\begin{tabular}{c|c}
			$k$ & Fiber type \\
			\hline
			$0$ & $I_0$ \\
			$1$ & $III$ \\
			$2$ & $I_0^*$ \\
			$3$ & $III^*$
		\end{tabular}
	\end{center}
\end{proposition}

\begin{proposition}
	Let $f \in \mathbb{F}_{3^{2n}}[t]$ be a non-square polynomial. Consider the elliptic surface $E_{4, \tau^3}$ given by
	$$E_{4, \tau}: \eta^2 = \xi^3 - f^3\xi.$$
	For $\alpha \in \overline{\mathbb{F}}_{3^{2n}}$, write $f(t) = (t - \alpha)^{\ell}\tilde{f}(t)$, where $\tilde{f}(\alpha) \neq 0$, and define $k = 3\ell - 4\lfloor \frac{3\ell}{4} \rfloor$ (in particular, $k \in \{0, 1, 2, 3\}$). For the fiber at infinity, we define
	$$\left\lceil \frac{3 \deg f}{4} \right\rceil = \frac{3 \deg f + k}{4}.$$
	The fibers at $\alpha$ and at $\infty$ are given in the following table.
	
	\begin{center}
		\begin{tabular}{c|c}
			$k$ & Fiber type \\
			\hline
			$0$ & $I_0$ \\
			$1$ & $III$ \\
			$2$ & $I_0^*$ \\
			$3$ & $III^*$
		\end{tabular}
	\end{center}
\end{proposition}

\vspace{0.3cm}

\noindent{\large \textbf{The sextic twists}}

\vspace{0.3cm}

\noindent Consider $\Phi_{-1, 1}$, which is an order $6$ morphism of $E$. We want a curve also admitting an order $6$ morphism. The characteristic of the fields we are working on is $3$, which divides $6$, so we cannot consider a Kummer extension as in \cite{BTT}. We first choose a non-square element $f \in \mathbb{F}_{3^{2n}}(\mathcal{X})$ and consider the field extension $\mathbb{F}_{3^{2n}}(\mathcal{X})(s)$ given by
$$s^2 = f.$$
For any element $g \in \mathbb{F}_{3^{2n}}(\mathcal{X})(s)$ which cannot be written as $h^3 - h$, the field extension $\mathbb{F}_{3^{2n}}(\mathcal{X})(s,w)$ given by
$$w^3 - w = g$$
will be of degree $6$ over $\mathbb{F}_{3^{2n}}(\mathcal{X})$. The morphism $\tau$ given by
$$\tau: (x, y, s, w) \mapsto (x + 1, -y, -s, w + 1)$$
has order $6$. In addition, since $w^3 - w$ is fixed by $\tau$, we must have $g \in \mathbb{F}_{3^{2n}}(\mathcal{X})(s)^{\tau} = \mathbb{F}_{3^{2n}}(\mathcal{X})$. Therefore, we consider the curve
$$\mathcal{H}: \left\{
\begin{array}{ll}
	s^2 & = f \\
	w^3 - w & = g
\end{array}
\right., \quad \textrm{with } f, g \in \mathbb{F}_{3^{2n}}(\mathcal{X}),$$
where $f$ is not a square and $g$ is not of the form $h^3 - h$ for any $h \in \mathbb{F}_{3^{2n}}(\mathcal{X})$.

Denoting again by $\mathbb{F}$ the fixed subfield of $\mathbb{F}_{3^{2n}}(\mathcal{X})(x, y, s, w)$ by $\Phi_{-1, 1} \times \tau$ yields
$$\mathbb{F}_{3^{2n}}(\mathcal{X})(s^2(x - w), s^3y) \subseteq \mathbb{F}.$$
One can obtain equality by proceeding as in the quartic twist case with the field extensions
$$\mathbb{F}_{3^{2n}}(\mathcal{X})(x - w, y, s) / \mathbb{F}_{3^{2n}}(\mathcal{X})(s^2(x - w), s^3y) \quad \textrm{ with basis } \{1, s\}$$
and
$$\mathbb{F}_{3^{2n}}(\mathcal{X})(x, y, s, w) / \mathbb{F}_{3^{2n}}(\mathcal{X})(x - w, y, s) \quad \textrm{ with basis } \{1, x, x^2\}.$$
Denoting $\xi = s^2(x - w)$ and $\eta = s^3y$ shows that the expression of the generic fiber of the elliptic surface given by $(E \times \mathcal{G}) / \langle \Phi_{-1, 1} \times \tau \rangle \rightarrow \mathcal{X}$ is
$$E_{6, \tau}: \eta^2 = \xi^3 - f^2\xi + f^3g.$$

\begin{remark}
	The morphisms $\Phi_{-1, 1} \times \tau^5$, $\Phi_{-1, -1} \times \tau$ and $\Phi_{-1, -1} \times \tau^5$ have also order $6$. They yield, respectively, the elliptic curves
	$$E_{6, \tau^5}: \eta^2 = \xi^3 - f^2\xi - f^3g,$$
	$E_{6, \tau^5}$ and $E_{6, \tau}$.
\end{remark}

\begin{proposition}\label{prop:sextic-twists}
	Let $\mathcal{X}$ be a curve defined over $\mathbb{F}_{3^{2n}}$ and let $f, g \in \mathbb{F}_{3^{2n}}(\mathcal{X})$ be such that $f$ is not a square and $g$ is not of the form $h^3 - h$ for any $h \in \mathbb{F}_{3^{2n}}(\mathcal{X})$. The elliptic curve $E / \mathbb{F}_3: y^2 = x^3 - x$ admits the sextic twists 
	$$E_{6, \tau}: \eta^2 = \xi^3 - f^2\xi + f^3g \quad \textrm{and} \quad E_{6, \tau^5}: \eta^2 = \xi^3 - f^2\xi - f^3g,$$
	which are isomorphic over the function field $\mathbb{F}_{3^{2n}}(\mathcal{X})$.
\end{proposition}
\begin{proof}
	The isomorphisms between these elliptic curves are given by
	$$\begin{array}{ccc}
		E & \cong & E_{6, \tau} \\
		(x, y) & \mapsto & (s^2(x - w), s^3y) \\
		(s^{-2}\xi + w, s^{-3}\eta) & \mapsfrom & (\xi, \eta)
	\end{array},$$
	$$\begin{array}{ccc}
		E & \cong & E_{6, \tau^5} \\
		(x, y) & \mapsto & (s^2(x + w), s^3y) \\
		(s^{-2}\xi - w, s^{-3}\eta) & \mapsfrom & (\xi, \eta)
	\end{array}$$
	and
	$$\begin{array}{ccc}
		E_{6, \tau} & \cong & E_{6, \tau^5} \\
		(\xi, \eta) & \mapsto & (-\xi, i\eta) \\
		(-\xi, -i\eta) & \mapsfrom & (\xi, \eta)
	\end{array},$$
	where $i \in \mathbb{F}_9$ is such that $i^2 = -1$. The first two isomorphisms are defined over a sextic extension of $\mathbb{F}_{3^{2n}}(\mathcal{X})$, while the last one is defined over $\mathbb{F}_{3^{2n}}(\mathcal{X})$.
\end{proof}

\begin{proposition}
	Let $\mathcal{X}$ be a curve defined over $\mathbb{F}_{3^{2n}}$, $\mathcal{C}$ be an Artin-Schreier extension of $\mathcal{X}$ given by $\mathcal{C}: w^3 - w = g$, for some $g \in \mathbb{F}_{3^{2n}}(\mathcal{X})$ which is not of the form $h^3 - h$ for any $h \in \mathbb{F}_{3^{2n}}(\mathcal{X})$. Consider $\mathcal{D}$ a Kummer extension of $\mathcal{X}$ given by $s^2 = f$, for some non-square element $f \in \mathbb{F}_{3^{2n}}(\mathcal{X})$ and $\mathcal{H}$ be a covering of $\mathcal{X}$ given by
	$$\mathcal{H}: \left\{
	\begin{array}{ll}
		s^2 & = f \\
		w^3 - w & = g
	\end{array}
	\right..$$
	Denote by $L_{E}(T)$ the $L$-function of the elliptic curve $E$ over $\mathbb{F}_{3^{2n}}$ and write the $L$-functions of $\mathcal{X}$, of $\mathcal{C}$, of $\mathcal{D}$ and of $\mathcal{H}$ as, respectively,
	$$L_{\mathcal{X}}(T) = L_{E}(T)^{m(\mathcal{X})} \cdot X(T), \quad L_{\mathcal{C}}(T) = L_{E}(T)^{m(\mathcal{C})} \cdot C(T),$$
	$$L_{\mathcal{D}}(T) = L_{E}(T)^{m(\mathcal{D})} \cdot D(T)$$
	and
	$$L_{\mathcal{H}}(T) = L_{E}(T)^{m(\mathcal{H})} \cdot H(T),$$
	with 
	$$\gcd(L_{E}, X) = \gcd(L_{E}, C) = \gcd(L_{E}, D) = \gcd(L_{E}, H) = 1,$$
	then 
	{\small $$\textrm{rank}\ E_{6, \tau}(\mathbb{F}_{3^{2n}}(\mathcal{X})) = \textrm{rank}\ E_{6, \tau^5}(\mathbb{F}_{3^{2n}}(\mathcal{X})) = 2[m(\mathcal{H}) - m(\mathcal{C}) - m(\mathcal{D}) + m(\mathcal{X})].$$}
\end{proposition}
\begin{proof}
	Note that $\rho(\xi, \eta) = (\xi + f, -\eta)$ is an order $6$ automorphism of $E_{6, \tau}$. We now fix $\zeta \in \mathbb{C}$ a $6$-th primitive root of unity. Define $V = E_{6, \tau}(\mathbb{F}_{3^{2n}}(\mathcal{H})) \otimes \mathbb{Q}$ and let $\zeta \cdot (P \otimes a) = \rho(P) \otimes a$. Since $\rho \circ \tau = \tau \circ \rho$, $V$ can be seen as a $\mathbb{Q}(\zeta)$-vector space. In this way, $\tau$ induces a linear map $\hat{\tau}$ in $V$ such that $\hat{\tau}^6 = \textrm{id}_V$, and $V$ can be decomposed into $6$ eigenspaces associated to the eigenvalues $1, \zeta, \ldots, \zeta^5$.
	
	\begin{itemize}
		\item $\textrm{Ker}(\hat{\tau} - \textrm{id}_V)$: Comparing coordinates, one sees that $0 \neq P \otimes a \in \textrm{Ker}(\hat{\tau} - \textrm{id}_V)$ if, and only if,
		$$P \otimes a = (x, y) \otimes a,$$
		with $(x, y) \in E_{6, \tau}(\mathbb{F}_{3^{2n}}(\mathcal{X}))$. Therefore, one obtains
		$$\textrm{Ker}(\hat{\tau} - \textrm{id}_V) = E_{6, \tau}(\mathbb{F}_{3^{2n}}(\mathcal{X})) \otimes \mathbb{Q}$$
		and
		$$\dim_{\mathbb{Q}} \textrm{Ker}(\hat{\tau} - \textrm{id}_V) = \textrm{rank }E_{6, \tau}(\mathbb{F}_{3^{2n}}(\mathcal{X})).$$
		
		\item $\textrm{Ker}(\hat{\tau} - \zeta \cdot \textrm{id}_V)$: This kernel is composed of the vectors $(x + fw, ys) \otimes a$, with $x, y \in \mathbb{F}_{3^{2n}}(\mathcal{X})$. Substituting in the equation of $E_{6, \tau}$ yields
		$$fy^2 = x^3 - f^2x - f^3g,$$
		which is equivalent to having $(f^{-1}x, f^{-1}y) \in E_{3, \tau}(\mathbb{F}_{3^{2n}}(\mathcal{X}))$, so
		$$\dim_{\mathbb{Q}} \textrm{Ker}(\hat{\tau} - \zeta \cdot \textrm{id}_V) = \textrm{rank }E_{3, \tau}(\mathbb{F}_{3^{2n}}(\mathcal{X})) = 2m(\mathcal{C}) - 2m(\mathcal{X}),$$
		where the last equality comes from Theorem \ref{teo:rank-cubic-twist}.		
		
		\item $\textrm{Ker}(\hat{\tau} - \zeta^2 \cdot \textrm{id}_V)$: This kernel is composed of the vectors $(x - fw, y) \otimes a$, with $x, y \in \mathbb{F}_{3^{2n}}(\mathcal{X})$. Substituting in the equation of $E_{6, \tau}$ yields
		$$y^2 = x^3 - f^2x.$$
		Hence
		$$\dim_{\mathbb{Q}} \textrm{Ker}(\hat{\tau} - \zeta^2 \cdot \textrm{id}_V) = \textrm{rank }E_2(\mathbb{F}_{3^{2n}}(\mathcal{X})) = 4m(\mathcal{D}) - 4m(\mathcal{X}),$$
		where $E_2$ is the quadratic twist given in Proposition \ref{prop:quadratic-twists} and the last equality comes from Theorem \ref{rank-quadratic-twist}.
		
		\item $\textrm{Ker}(\hat{\tau} - \zeta^3 \cdot \textrm{id}_V)$: This kernel is composed of the vectors $(x, ys) \otimes a$, with $x, y \in \mathbb{F}_{3^{2n}}(\mathcal{X})$. An argument analogous to the one we used in $\textrm{Ker}(\hat{\tau} - \zeta \cdot \textrm{id}_V)$ allows us to conclude that
		$$\dim_{\mathbb{Q}} \textrm{Ker}(\hat{\tau} - \zeta \cdot \textrm{id}_V) = \textrm{rank }E_{3, \tau^2}(\mathbb{F}_{3^{2n}}(\mathcal{X})) = 2m(\mathcal{C}) - 2m(\mathcal{X}).$$
		
		\item $\textrm{Ker}(\hat{\tau} - \zeta^4 \cdot \textrm{id}_V)$: This kernel is composed of the vectors $(x + fw, y) \otimes a$. Substituting in the equation of $E_{6, \tau}$ yields $(x, y) \in E_{6, \tau^5}(\mathbb{F}_{3^{2n}}(\mathcal{X}))$ and 
		$$\dim_{\mathbb{Q}} \textrm{Ker}(\hat{\tau} - \textrm{id}_V) = \textrm{rank }E_{6, \tau^5}(\mathbb{F}_{3^{2n}}(\mathcal{X})) = \textrm{rank }E_{6, \tau}(\mathbb{F}_{3^{2n}}(\mathcal{X})),$$
		where the last equality comes from the fact that $E_{6, \tau}$ and $E_{6, \tau^5}$ are isomorphic over $\mathbb{F}_{3^{2n}}(\mathcal{X})$.
		
		\item $\textrm{Ker}(\hat{\tau} - \zeta^5 \cdot \textrm{id}_V)$: Finally, this kernel is composed of the vectors $(x - fw, ys)$. An argument analogous to the one we used in $\textrm{Ker}(\hat{\tau} - \zeta \cdot \textrm{id}_V)$ allows us to conclude that
		$$\dim_{\mathbb{Q}} \textrm{Ker}(\hat{\tau} - \zeta^5 \cdot \textrm{id}_V) = \textrm{rank }E(\mathbb{F}_{3^{2n}}(\mathcal{X})) = 4m(\mathcal{X}).$$
	\end{itemize}
	
	The result now follows by observing that, as before, $\dim V = 4m(\mathcal{H})$.
\end{proof}

\begin{corollary}
	Let $\mathcal{X}$ be a curve defined over $\mathbb{F}_{3^{2n}}$, $\mathcal{C}$ be an Artin-Schreier extension of $\mathcal{X}$ given by $\mathcal{C}: w^3 - w = g$, for some $g \in \mathbb{F}_{3^{2n}}(\mathcal{X})$ which is not of the form $h^3 - h$ for any $h \in \mathbb{F}_{3^{2n}}(\mathcal{X})$. Consider $\mathcal{D}$ a Kummer extension of $\mathcal{X}$ given by $s^2 = f$, for some non-square element $f \in \mathbb{F}_{3^{2n}}(\mathcal{X})$ and $\mathcal{H}$ be an extension of $\mathcal{X}$ given by 
	$$\mathcal{H}: \left\{
	\begin{array}{ll}
		s^2 & = f \\
		w^3 - w & = g
	\end{array}
	\right..$$
	If $\mathcal{H}$ is maximal (resp. minimal) over $\mathbb{F}_{3^{2n}}$ for odd (resp. even) $n$ and $g(\mathcal{X})$, $g(\mathcal{C})$, $g(\mathcal{D})$ and $g(\mathcal{H})$ denote, respectively, the genera of the curves $\mathcal{X}$, $\mathcal{C}$, $\mathcal{D}$ and $\mathcal{H}$, then
	{\small $$\textrm{rank}\ E_{6, \tau}(\mathbb{F}_{3^{2n}}(\mathcal{X})) = \textrm{rank}\ E_{6, \tau^5}(\mathbb{F}_{3^{2n}}(\mathcal{X})) = 2[g(\mathcal{H}) - g(\mathcal{C}) - g(\mathcal{D}) + g(\mathcal{X})].$$}
\end{corollary}

The proof of the following proposition is a direct application of Tate's algorithm, after some changing analogous to those we used in Proposition \ref{fiber-quadratic}, and hence omitted.

\begin{proposition}
	Let $f \in \mathbb{F}_{3^{2n}}[t]$ be a non-square polynomial and $g \in \mathbb{F}_{3^{2n}}[t]$ be a polynomial such that $g \neq h^3 - h$ for any polynomial $h$. Consider the elliptic surfaces $E_{6, \tau}$ and $E_{6, \tau^5}$ given by
	$$E_{6, \tau}: \eta^2 = \xi^3 - f^2\xi + f^3g \quad \textrm{and} \quad E_{6, \tau^5}: \eta^2 = \xi^3 - f^2\xi - f^3g.$$
	Then, for $\alpha \in \overline{\mathbb{F}}_{3^{2n}}$,
	\begin{itemize}
		\item The fiber at $t = \alpha$ is of type $I_0$ if the multiplicity of $\alpha$ as a root of $f$ is even or $0$;
		\item The fiber at $t = \alpha$ is of type $I_0^*$ if the multiplicity of $\alpha$ as a root of $f$ is odd;
	\end{itemize}
	For the fiber at infinity, we can assume that $\deg g \not\equiv 0$ (mod $3$). We then write
	$$n = \left\lceil \frac{3 \deg f + \deg g}{6} \right\rceil = \frac{3 \deg f + \deg g + k}{6},$$
	with $k \in \{1, 2, 4, 5\}$, and the fiber at infinity is given in the following table.
	
	\begin{center}
		\begin{tabular}{c|c}
			$k$ & Fiber type \\
			\hline
			$1$ & $II$ \\
			$2$ & $IV$ \\
			$4$ & $IV^*$ \\
			$5$ & $II^*$
		\end{tabular}
	\end{center}
\end{proposition}

\section{Acknowledgments}
This work was developed during the author's PhD studies at the Instituto de Ciências Matemáticas e Computação of the Universidade de São Paulo and at the Bernoulli Institute of the Rijksuniversiteit Groningen. The author was supported by CNPq (Brazil), grant 140589/2021-0, by CAPES (Brazil), finance code 001.

The author is thankful to Jaap Top for his suggestions on how to improve this work and for proofreading this paper.


\end{document}